\documentclass[a4paper,11pt]{article}

\usepackage{amsfonts,graphicx,subfig,amsmath,amssymb,enumitem,
,url,color,authblk,amsthm,multicol,cite,float,bigints}

\DeclareMathOperator\erf{erf}

\theoremstyle{definition}

\theoremstyle{theorem}
\newtheorem{lemma}{Lemma}[section]
\theoremstyle{theorem}
\newtheorem{theorem}{Theorem}[section]
\theoremstyle{corollary}
\newtheorem{corollary}{Corollary}[section]
\theoremstyle{remark}
\newtheorem{remark}{Remark}

\title{Existence and uniqueness of the modified error function}

\author[1,2]{Andrea N. Ceretani\thanks{aceretani@austral.edu.ar}}
\author[3,4]{Natalia N. Salva\thanks{natalia.salva@yahoo.com.ar}}
\author[1]{Domingo A. Tarzia\thanks{dtarzia@austral.edu.ar}}
\affil[1]{{\small CONICET - Depto. Matem\'atica, Facultad de Ciencias Empresariales, Univ. Austral, Paraguay 1950, S2000FZF Rosario, Argentina.}}
\affil[2]{{\small Depto. de Matem\'atica, Facultad de Ciencias Exactas, Ingenier\'ia y Agrimensura, Univ. Nacional de Rosario, Pellegrini 250, S2000BTP Rosario, Argentina.}}
\affil[3]{CONICET - CNEA, Depto. de Mec\'anica Computacional, Centro At\'omico Bariloche, Av. Bustillo 9500, 8400 Bariloche, Argentina.}
\affil[4]{Depto. de Matem\'atica, Centro Regional Bariloche, Univ. Nacional del Comahue, Quintral 250, 8400 Bariloche, Argentina.}
\date{}

\begin{document}  
\maketitle

\begin{abstract}
This article is devoted to prove the existence and uniqueness of solution to the non-linear second order differential problem through which is defined the modified error function introduced in {\em Cho-Sunderland, J. Heat Transfer, 96-2:214-217, 1974}. We prove here that there exists a unique non-negative analytic solution for small positive values of the parameter on which the problem depends.
\end{abstract}



\section{Introduction}\label{Sec:Intro}
In 1974, Cho and Sunderland \cite{ChSu1974} studied a solidification process with temperature-dependent thermal conductivity and obtained an explicit similarity solution in terms of what they called a {\em modified error function}. This function is defined as the solution to the following non-linear differential problem:
\begin{subequations}\label{Pb:y}
\begin{align}
\label{eq:y}&[(1+\delta y(x))y'(x)]'+2xy'(x)=0\quad 0<x<+\infty\\
\label{cond:0}&y(0)=0\\
\label{cond:infty}&y(+\infty)=1
\end{align}
\end{subequations}
where $\delta\geq -1$ is given. Graphics for numerical solutions of (\ref{Pb:y}) for different values of $\delta$ can be found in \cite{ChSu1974}. The classical error function is defined by:
\begin{equation}\label{erf}
\erf(x)=\frac{2}{\sqrt{\pi}}\displaystyle\int_0^x\exp(-z^2)dz,\quad x>0,
\end{equation}
and it is a solution to (\ref{Pb:y}) when $\delta=0$. This makes meaningful the denomination {\em modified error function} given for the solution to problem (\ref{Pb:y}).

The modified error function has also appeared in the context of diffusion problems before 1974 \cite{Cr1956,Wa1950}. It was also used later in several opportunities to find similarity solutions to phase-change processes \cite{BrNaTa2007,FrVi1987,Lu1991,OlSu1987,SaTa2011-b,Ta1998}. It was cited in \cite{CoKa1994}, were several non-linear ordinary differential problems arise from a wide variety of fields are presented. Closed analytical solutions for Stefan problems with variable diffusivity is given in \cite{VoFa2013}. Temperature-dependent thermal coefficients are very important in thermal analysis, e.g. see \cite{So2016}. 
Nevertheless, to the knowledge of the authors, the existence and uniqueness of the solution to problem (\ref{Pb:y}) has not been yet proved. This article is devoted to prove it for small $\delta>0$ using a fixed point strategy.

\section{Existence and uniqueness of solution to problem (\ref{Pb:y})}\label{Sec:EyU}
The main idea developed in this Section is to study problem (\ref{Pb:y}) through the linear problem given by the differential equation:
\begin{equation}\label{eq:yLin}
[(1+\delta\Psi_h(x))y'(x)]'+2xy'(x)=0, \quad 0<x<+\infty, \tag{\ref{eq:y}$^\star$}
\end{equation}
and conditions (\ref{cond:0}), (\ref{cond:infty}). The function $\Psi_h$ in (\ref{eq:yLin}) is defined by:
\begin{equation}\label{Psi}
\Psi_h(x)=1+\delta h(x),\quad x>0,
\end{equation}
where $\delta>0$, $h\in K\subset X$ is given and:
\begin{subequations}\label{KX}
\begin{align}
\label{X}&X=\left\{h:\mathbb{R}_0^+\to\mathbb{R}\,/\,h\text{ is an analytic function},\,||h||_{\infty}<\infty\right\}\\
\label{K}&K=\left\{h\in X\,/\, ||h||_{\infty}\leq 1,\,0 \leq h,\,h(0)=0,\,h(+\infty)=1\right\}.
\end{align}
\end{subequations}
Hereinafter, we will refer to the problem given by (\ref{eq:yLin}), (\ref{cond:0}) and (\ref{cond:infty}) as problem (\ref{Pb:y}$^\star$). Let us observe that $K$ is non-empty closed subset of the Banach space $X$.

The advantage in considering the linear equation (\ref{eq:yLin}) is that it can be easily solved through the substitution $v=y'$. Thus, we have the following result:

\begin{theorem}\label{Th:caract1}
Let $h\in K$ and $\delta>0$. The solution $y$ to problem (\ref{Pb:y}$^\star$) is given by:
\begin{equation}\label{yh}
y(x)=C_h\displaystyle\int_0^x \frac{1}{\Psi_h(\eta)}\exp\left(-2\displaystyle\int_0^\eta\frac{\xi}{\Psi_h(\xi)}d\xi\right)d\eta\quad x\geq 0,
\end{equation}
where the constant $C_h$ is defined by:
\begin{equation}\label{Ch}
C_h=\left(\displaystyle\int_0^{+\infty} \frac{1}{\Psi_h(\eta)}\exp\left(-2\displaystyle\int_0^\eta\frac{\xi}{\Psi_h(\xi)}d\xi\right)d\eta\right)^{-1}.
\end{equation}
\end{theorem}

\begin{proof}
Let us first observe that the constant $C_h$ given by (\ref{Ch}) is well defined, that is, that $C_h\in\mathbb{R}$. In fact, we have:
\begin{equation}\label{Ch-1}
\begin{split}
|C_h^{-1}|&=\displaystyle\int_0^{+\infty}\frac{1}{\Psi_h(\eta)}\exp\left(-2\displaystyle\int_0^\eta\frac{\xi}{\Psi(\xi)}d\xi\right)d\eta\\
&\geq\frac{1}{1+\delta}\displaystyle\int_0^{+\infty}\exp(-\eta^2)d\eta=\frac{\sqrt{\pi}}{2(1+\delta)}
\end{split}
\end{equation}
Now the proof follows easily by checking that the function $y$ given by (\ref{yh}) satisfies problem (\ref{Pb:y}$^\star$).
\end{proof}

The following result is an immediate consequence of Theorem \ref{Th:caract1}.

\begin{corollary}\label{Co:caract2}
Let $y\in K$ and $\delta>0$. Then $y$ is a solution to problem (\ref{Pb:y}) if and
only if $y$ is a fixed point of the operator $\tau$ from $K$ to $X$ defined by:
\begin{equation}
\tau (h)(x)=C_h\displaystyle\int_0^x \frac{1}{\Psi_h(\eta)}\exp\left(-2\displaystyle\int_0^\eta\frac{\xi}{\Psi_h(\xi)}d\xi\right)d\eta\quad x>0,
\end{equation}
with $C_h$ given by (\ref{Ch}).
\end{corollary}

\begin{remark}
Observe that $\tau(K)\subset K$.
\end{remark}

We will now focus on analyzing when $\tau$ has only one fixed point. The estimations summarized next will be useful in the following.

\begin{lemma}\label{Le:cotas}
Let $h,h_1,h_2\in K$, $\delta>0$ and $x\geq 0$. We have:
\begin{enumerate}
\item[]
\item[a)] $\displaystyle\bigintsss_0^x\left|
\frac{\exp\left(-2\displaystyle\int_0^\eta\frac{\xi}{\Psi_{h_1}(\xi)}d\xi\right)}{\Psi_{h_1}(\eta)}-
\frac{\exp\left(-2\displaystyle\int_0^\eta\frac{\xi}{\Psi_{h_2}(\xi)}d\xi\right)}{\Psi_{h_2}(\eta)}\right|d\eta$
\item[]
\item[]$\leq\frac{\sqrt{\pi}}{4}\delta\sqrt{1+\delta}(3+\delta)||h_1-h_2||_{\infty}$,
\item[]
\item[b)] $|C_{h_1}-C_{h_2}|\leq \frac{1}{\sqrt{\pi}}\delta\sqrt{1+\delta}(1+\delta)^2(3+\delta)||h_1-h_2||_{\infty}$,
\item[]
\item[c)] $\displaystyle\int_0^x\frac{1}{\Psi_h(\eta)}\exp\left(-2\displaystyle\int_0^\eta\frac{\xi}{\Psi_h(\xi)}d\xi\right)d\eta\leq\frac{\sqrt{\pi(1+\delta)}}{2}$.
\end{enumerate}
\end{lemma}

\begin{proof}
Let $f$ be the real function defined on $\mathbb{R}^+_0$ by $f(x)=\exp(-2x)$. If $h_1\leq h_2$, it follows from the Mean Value Theorem applied to function $f$ that:
\begin{equation}
\begin{split}
&\left|
\exp\left(-2\displaystyle\int_0^\eta\frac{\xi}{\Psi_{h_1}(\xi)}d\xi\right)-
\exp\left(-2\displaystyle\int_0^\eta\frac{\xi}{\Psi_{h_2}(\xi)}d\xi\right)\right|\\
&=2\exp\left(-2\displaystyle\int_0^\eta\frac{\xi}{\Psi_{h_3}(\xi)}d\xi\right)\left|\displaystyle\int_0^\eta\frac{\xi}{\Psi_{h_1}(\xi)}d\xi-\displaystyle\int_0^\eta\frac{\xi}{\Psi_{h_2}(\xi)}d\xi\right|\\
&\leq \delta||h_2-h_1||_{\infty}\eta^2\exp\left(\frac{-\eta^2}{1+\delta}\right),
\end{split}
\end{equation}
where $h_1\leq h_3\leq h_2$. Now a) follows from regular computations. When $h_1\not\leq h_2$, as the LHS in a) can be bounded for the same expression but applied to $h_m=\min\{h_1,h_2\}$ and $h_M=\max\{h_1,h_2\}$, the proof runs as before and it is completed having into consideration that $||h_1-h_2||_{\infty}=||h_M-h_m||_{\infty}$.

The proof of b) follows from a), and c) can be obtained from regular  computations.

%\begin{equation}
%\begin{split}
%&\left|
%\exp\left(-2\displaystyle\int_0^\eta\frac{\xi}{\Psi_{h_1}(\xi)}d\xi\right)-
%\exp\left(-2\displaystyle\int_0^\eta\frac{\xi}{\Psi_{h_2}(\xi)}d\xi\right)\right|\\
%&\leq \left|\exp\left(-2\displaystyle\int_0^\eta\frac{\xi}{\Psi_M(\xi)}d\xi\right)-
%\exp\left(-2\displaystyle\int_0^\eta\frac{\xi}{\Psi_m(\xi)}d\xi\right)\right|,
%\end{split}
%\end{equation}

\end{proof}

\begin{theorem}\label{Th:contrac}
Let $\delta_1>0$ be the only one positive solution to the equation:
\begin{equation}\label{eq:delta1}
\frac{x}{2}(1+x)^{3/2}(3+x)[1+(1+x)^{3/2}]=1.
\end{equation}
If $0<\delta<\delta_1$, then $\tau$ is a contraction.
\end{theorem}

\begin{proof}
Let $g$ be the real function defined by:
\begin{equation}\label{g}
g(x)=\frac{x}{2}(1+x)^{3/2}(3+x)[1+(1+x)^{3/2}]\quad x\geq 0.
\end{equation}
Since $g$ is an increasing function from 0 to $+\infty$, we have that equation (\ref{eq:delta1}) admits only one positive solution $\delta_1$.

Let be now $h_1,h_2\in K$ and $x\geq 0$. From Lemma \ref{Le:cotas}, (\ref{Ch-1}) and:
\begin{equation*}
\begin{split}
&|\tau (h_1)(x)-\tau (h_2)(x)|\\
&\leq C_{h_1}\displaystyle\bigintsss_0^x\left|
\frac{\exp\left(-2\displaystyle\int_0^\eta\frac{\xi}{\Psi_{h_1}(\xi)}d\xi\right)}{\Psi_{h_1}(\eta)}-
\frac{\exp\left(-2\displaystyle\int_0^\eta\frac{\xi}{\Psi_{h_2}(\xi)}d\xi\right)}{\Psi_{h_2}(\eta)}\right|d\eta\\
&+|C_{h_1}-C_{h_2}|\displaystyle\int_0^x\frac{1}{\Psi_{h_2}(\eta)}\exp\left(-2\displaystyle\int_0^\eta\frac{\xi}{\Psi_{h_2}(\xi)}d\xi\right)d\eta,
\end{split}
\end{equation*}
it follows that $||\tau (h_1)-\tau (h_2)||_{\infty}\leq \gamma||h_1-h_2||_{\infty}$, where $\gamma=g(\delta)$. Recalling that $g$ is an increasing function, it follows that $\tau$ is a contraction when $0<\delta<\delta_1$.
\end{proof}

\begin{remark}
From a numerical computation, it can be found that:
\begin{equation*}
0.203701<\delta_1<0.203702.
\end{equation*}
\end{remark}

We are now in the position to formulate our main result:

\begin{corollary}\label{Co:EyU}
Let $\delta_1$ be as in Theorem \ref{Th:contrac}. If $0<\delta<\delta_1$, then problem (\ref{Pb:y}) has a unique non-negative analytic solution.
\end{corollary}

\begin{proof}
It is a direct consequence of Corollary \ref{Co:caract2}, Theorem \ref{Th:contrac} and the Banach Fixed Point Theorem.
\end{proof}

\subsection*{Acknowledgments}
This paper has been partially sponsored by the Project PIP No. 0534 from CONICET-UA (Rosario, Argentina) and AFOSR-SOARD Grant FA 9550-14-1-0122.

\bibliographystyle{plain}
\bibliography{References}

\end{document}